\documentclass{article}
\usepackage{amsmath,amsgen,amstext,amsbsy,amssymb,amsopn,amsfonts,amsthm}
\usepackage{graphics,graphicx,color,rotating,lscape,afterpage,subfigure,chngcntr}
\usepackage{authblk,scalerel}

\graphicspath{{graphics/}}
\counterwithin{figure}{section}

\newtheorem{Def}{Definition}[section]
\newtheorem{Thm}[Def]{Theorem}
\newtheorem{Prp}[Def]{Proposition}
\newtheorem{Lem}[Def]{Lemma}
\newtheorem{Cor}[Def]{Corollary}

\def\ar{\scalerel*{\includegraphics{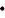}}{]}}
\def\arx{\scalerel*{\includegraphics{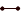}}{]}}
\def\arxx{\scalerel*{\includegraphics{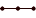}}{]}}
\def\arxxn{\scalerel*{\includegraphics{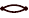}}{]}}
\def\arxxx{\scalerel*{\includegraphics{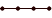}}{]}}
\def\arxxxn{\scalerel*{\includegraphics{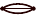}}{]}}
\def\arxxnx{\scalerel*{\includegraphics{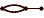}}{]}}
\def\arxxnxn{\scalerel*{\includegraphics{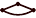}}{]}}
\def\arxxxx{\scalerel*{\includegraphics{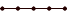}}{]}}
\def\arxxxxn{\scalerel*{\includegraphics{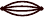}}{]}}
\def\arxxxnx{\scalerel*{\includegraphics{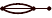}}{]}}
\def\arxxxnxn{\scalerel*{\includegraphics{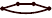}}{]}}
\def\arxxnxx{\scalerel*{\includegraphics{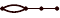}}{]}}
\def\arxxnxxn{\scalerel*{\includegraphics{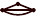}}{]}}
\def\arxxnxnx{\scalerel*{\includegraphics{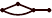}}{]}}
\def\arxxnxnxn{\scalerel*{\includegraphics{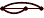}}{]}}
\def\arxxuxx{\scalerel*{\includegraphics{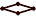}}{]}}
\def\arxxuxxn{\scalerel*{\includegraphics{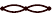}}{]}}

\begin{document}

\title{Ehrhart Polynomials of Generic Orthotopes}

\author[$\star$]{David Richter}

\affil[$\star$]{Department of Mathematics, Western Michigan University, Kalamazoo MI 49008 (USA),
\texttt{david.richter@wmich.edu}}

\maketitle

\begin{abstract}

A generic orthotope is an orthogonal polytope whose tangent cones are described
by read-once Boolean functions.  
The purpose of this note is to develop a theory of
Ehrhart polynomials for integral generic orthotopes.
The most remarkable part of this theory is a relation between
the number of lattice points in an integral generic orthotope $P$ and the number
of unit cubes in $P$ of various floral types.  This formula is facilitated
through the introduction of a set of ``local polynomials'' defined for 
every read-once Boolean function.
\end{abstract}


\section{Introduction}

This author introduced a theory of generic orthotopes in \cite{richter_genericorthotopes}.
The goal of this note is to
establish a theory of Ehrhart polynomials specialized to integral generic orthotopes, extending
the results of \cite{richter_genericorthotopes}.
The most prominent result here is a formula, stated as Theorem \ref{mainresult}, 
which demonstrates that counting
lattice points in a generic orthotope $P$ is equivalent to counting integral unit cubes
in $P$ while keeping track of their ``floral types''.
We demonstrate this formula by introducing an
algebra $\mathcal{F}$ generated by a set of equivalence classes of read-once
boolean functions (essentially series-parallel diagrams) and a couple of useful
bases $\{h(e_\alpha)\}$ and $\{h^{-1}(e_\alpha)\}$ for this algebra. 
This analysis gives rise to
a plethora of identities concerning integral generic orthotopes.
It also showcases the calculus of series-parallel diagrams used for studying
generic orthotopes as introduced in \cite{richter_genericorthotopes}.

A major theme of this work (including \cite{richter_genericorthotopes}) is that
while many workers have studied aspects
of orthogonal polytopes in a fixed dimension, usually $d=2$ or $d=3$, relatively
few have considered common structural questions among orthogonal polytopes irrespective
of dimension.  As one may notice,
we may prove all of the results shown here and in 
\cite{richter_genericorthotopes} in an elementary manner.
Thus, the results shown here appear among the lowest-hanging fruit in the 
deep, dark, and largely unexplored forest of general orthogonal polytopes.
On another point, this author wants to persuade the reader 
that generic orthotopes in particular possess a distinctive intrinsic charm and hopes that they
will notice this throughout this theory; for example, this author is impressed at the way that
ideas and results tend to flow as soon as one demands that every vertex of an orthogonal
polytope be floral.
%For example, manually working several
%of the computations described herein often brings one a sense of guilty pleasure,
%and this author thus encourages the reader to try their own hand at it.
(On this point, it should
be obvious that this author is biased!)

We have attempted to be as pedantic as possible in this note.
Thus, as we develop our adaptation of Ehrhart polynomials to generic orthotopes,
we concurrently walk the reader through a running example in 2 dimensions.
The main reason for this
is that we wish to highlight as brightly as possible the distinctiveness of
the Ehrhart theory as it pertains to generic orthotopes in particular.
Indeed, the underlying algebra $\mathcal{F}$ and the polynomials $h(e_\alpha)$ or $h^{-1}(e_\alpha)$
which are instrumental in expressing our results have been not found among any extant
works about Ehrhart theory and lattice point enumeration.  (This should not be surprising
because the introduction of read-once Boolean functions in the study of orthogonal polytopes
appears to be novel.
For that matter, we have not found these notions among the extant literature on Boolean functions or Boolean
algebras.)
The example serves to display the motivation of the theory but also to
better understand the notation and concepts used throughout.  
We also offer
an example in 3 dimensions in a separate section, should the reader wish
to see something more illustrative or more substantial.  Also,
this author feels that the examples also serve to better understand the foundational article
\cite{richter_genericorthotopes}.

\subsection{Some notions and conventions}

This subsection summarizes some notation conventions and notions used throughout.
The reader should refer to \cite{richter_genericorthotopes} for supplemental details.

We use the symbol $d$ 
to denote the dimension of the ambient space, and in most cases assume $d$ is 
fixed; certainly $d$ is fixed when we speak of a particular orthogonal
polytope.  Accordingly, we denote $\left[d\right]:=\{1,2,3,...,d\}$.
In most cases, we use the symbol $P$ for an arbitrary $d$-dimensional bounded
integral generic orthotope.

A ``congruence class'' may refer to a floral arrangement congruence class or a floral vertex
congruence class.  In all cases, we use ``congruence'' to mean Euclidean-geometric
congruence.  As we explained in \cite{richter_genericorthotopes}, the Coxeter group $BC_d$
acts on floral arrangements, so we say that $\alpha$ is congruent to $\alpha'$ if
there is a group element $g\in BC_d$ such that $\alpha'=g\cdot\alpha$.  In other words,
$\alpha$ is congruent to $\alpha'$ if both $\alpha$ and $\alpha'$ lie in the same orbit
under the standard action of $BC_d$ on $\mathbb{R}^d$.  Notice that the phrase
``floral arrangement congruence class'' is meaningful only when $d$ is fixed.
We use lower case Greek letters $\alpha,\beta$ to denote 
either a floral vertex, a floral arrangement, a floral vertex congruence class, or a floral
arrangement congruence class.  We justify the overuse of this notation by the 
fact that this theory is already laden with notation; we try to explain the usage 
where it may be unclear.

The {\it standard $d$-dimensional unit cube in $\mathbb{R}^d$} is $I^d=[0,1]^d$.  
A {\it $d$-dimensional integral unit cube} is any translate $v+I^d$, where $v\in\mathbb{Z}^d$.
If $k\in\{0,1,2,...,d\}$, then a {\it $k$-dimensional integral unit cube} is any $k$-dimensional
face of a $d$-dimensional integral unit cube.
A {\it relatively open $k$-dimensional integral unit cube} is the relative interior
of a $k$-dimensional integral unit cube with respect to its affine hull.
Notice that every $x\in\mathbb{R}^d$ belongs to a unique relatively open $k$-dimensional
integral unit cube
for some $k\in\{0,1,2,...,d\}$.  The dimension of the cube containing $x=(x_1,x_2,...,x_d)$ is given by
$$k(x):=\#\left(\{i:x_i\notin\mathbb{Z}\}\right).$$
Thus, $k(x)=0$ precisely when $x\in\mathbb{Z}^d$ is a lattice point
and $k(x)=d$ precisely when $x$ lies interior to a $d$-dimensional integral unit cube.

On multiple occasions throughout, we refer to the ``floral type'' of a point or a
relatively open cube.  The following, which may be inferred from the results in
\cite{richter_genericorthotopes}, is intended to clarify the meaning of this.

\begin{Prp}
Suppose $P\subset\mathbb{R}^d$ is a $d$-dimensional integral generic orthotope and
$C\subset P$ is a relatively open integral unit cube.  Then there is a 
floral arrangement $\alpha\subset\mathbb{R}^d$ 
such that the tangent cone at every point of $C$ is congruent to $\alpha$.
\end{Prp}

Thus, if $C\subset P$ is a relatively open integral unit cube, then define the {\it floral
type} of $C$ as the floral arrangement congruence class which contains
the tangent cone of any given point of $C$.  Notice that this meaning extends to the case
when $C$ is a point (i.e.\ a cube of dimension zero) of $P$.
Figure \ref{floraltypes} displays a 3-dimensional generic orthotope whose faces are marked with
various floral types.

We recall from \cite{richter_genericorthotopes} 
that we associate a {\it dimension} to every floral vertex $\alpha$, defined as the number of segments in 
the series-parallel diagram which defines it.  Thus, $\dim(\ar)=0$, $\dim(\arx)=1$, $\dim(\arxx)=\dim(\arxxn)=2$, 
and so on.
On several occasions, we also refer to the {\it degree} 
$\deg\alpha$ of a floral arrangement, this being its degree of genericity as defined
in \cite{richter_genericorthotopes}.  Moreover, if $\alpha$ is a floral arrangement, then
there exists a positive integer $k$ and a floral vertex $\beta$ supported on $k$ half-spaces 
such that $\alpha$ is congruent to $\beta\times\mathbb{R}^{d-k}$.  In such a circumstance,
we may rely on the relation $\deg\alpha=d-\dim\beta=d-k$.

\begin{figure} 
\centering 
\includegraphics[width=\textwidth]{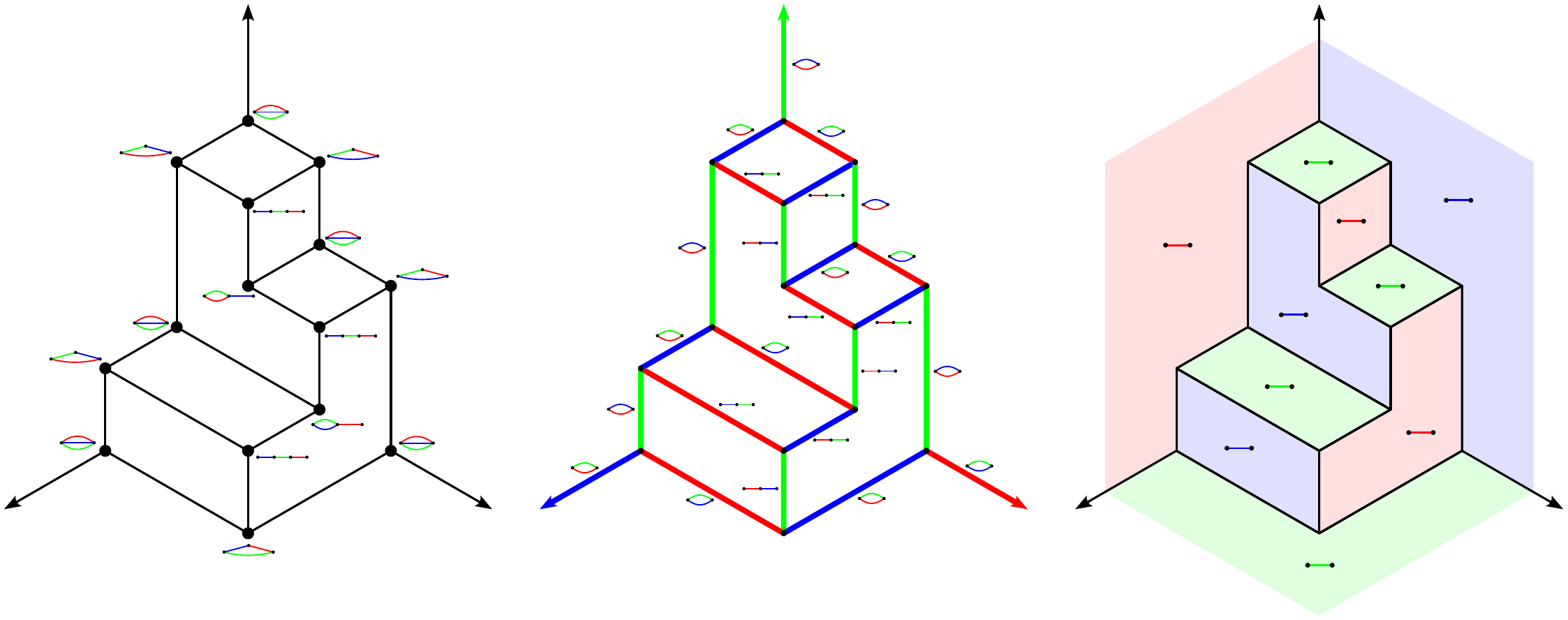}
\caption{Floral types for a 3-dimensional generic orthotope.}
\label{floraltypes}
\end{figure}

Suppose $s:\left[d\right]\rightarrow\{-1,0,1\}$
is a vector.  Define the {\it generalized orthant of $s$} as
$$\Omega_s:=\{(x_1,x_2,...,x_d):\mathrm{sign}(x_i)=s_i\hbox{ for all }i\}.$$
Notice $\Omega_s$ is homeomorphic to an open cell $\mathbb{R}^{k(s)}$, where
$k(s)=\#\left\{s_i:s_i\neq 0\right\}$, and $\mathbb{R}^d$ contains precisely 
$2^k\cdot\binom{d}{k}=2^k\cdot\frac{d!}{k!(d-k)!}$ generalized orthants
of dimension $k$.  Also, each $x=(x_1,x_2,...,x_d)\in\mathbb{R}^d$ lies in the generalized
orthant $\Omega_s$, where $s_i=\mathrm{sign}(s_i)\in\{-1,0,1\}$ for all $i$.

\section{Introductory theory and example}

If $P\subset\mathbb{R}^d$ is a bounded subset, then define
the {\it lattice point enumerator} of $P$ by
$$L(P):=\#(P\cap\mathbb{Z}^d).$$
If $P$ is an integral convex polytope and $tP$ is obtained
by uniformly dilating $P$ by the factor of $t\in\mathbb{N}$, then $L(tP)$ is well
known to be a polynomial function of $t$ called the Ehrhart polynomial of $P$,
cf.\ \cite{BR_2015,MS_2005}.
This note adapts this theory to the case when $P$ is an integral generic orthotope.

Throughout this note we use an example in 2 dimensions to illustrate the theory.
Thus, define an orthogonal polygon by 
$P:=\bigcup_{v\in S}\left(v+[0,1]^2\right),$
where
$$S=\left\{
\begin{array}{l}
(0,0),(0,1),(0,2),(0,3),(1,3),(2,1),(2,2), \\
(2,3),(3,1),(3,2),(3,3),(4,0),(4,1),(4,2), \\
(5,2),(6,1),(6,2),(6,3),(6,4) \\
\end{array}
\right\}.$$
Figure \ref{orthogoncolor} displays a sketch of $P$.  Apparently $P$ is an integral
generic orthotope with $L(P)=37$.

If $\alpha$ is a floral arrangement congruence class, then let
$L_\alpha(P)$ denote the number of lattice points $x\in P\cap\mathbb{Z}^d$ such that the
tangent cone at $x$ lies in $\alpha$.
For example, $L_{\ar}(P)$ counts the number of lattice points lying interior
to $P$ and $L_\alpha(P)$ counts the number of vertices of $P$ congruent to a member of
$\alpha$ when $\alpha$ is a floral vertex congruence class.
In the example, we distinguish the lattice points of various types by color.  Thus,
$$\begin{array}{rclc}
L_{\ar}(P) &=& 3 & \hbox{(green),} \\
L_{\arx}(P) &=& 16 & \hbox{ (black),} \\
L_{\arxx}(P) &=& 11 & \hbox{ (blue),} \\
L_{\arxxn}(P) &=& 7 & \hbox{ (red).} \\
\end{array}$$

If $t$ is a positive integer, then let $tP$ denote the image of $P$
under the dilation map $x\mapsto tx$.
That $L(tP)$ is a polynomial function of $t$ follows from well-known
results concerning Ehrhart polynomials.  Thus, 
we shall refer to $L(tP)$ as the {\it Ehrhart polynomial} of $P$.
We are interested in decomposing $L(tP)$ according to floral arrangement 
congruence classes.  Thus, we write
$$L(tP)=\sum_{\alpha}L_{\alpha}(tP),$$
where we sum over floral arrangement congruence classes $\alpha$.
As $L_\alpha(tP)$ is a polynomial function of $t$ for each $\alpha$,
let $L_{\alpha,k}(P)$ denote the coefficients such that
$$L_{\alpha}(tP)=\sum_k L_{\alpha,k}(P)t^k.$$
The coefficients $L_{\alpha,k}(P)$ for the running example appear in Figure 
\ref{lpenumerator}.  Thus, 
$$L_{\ar}(tP)=1-17t+19t^2\hbox{ and }L_{\arx}(tP)=-18+34t,$$
while $L_{\arxx}(tP)=11$ and $L_{\arxxn}(tP)=7$ are constant.
Notice that evaluating $L_{\alpha}(tP)$ at $t=1$ yields the values $L_{\alpha}(P)$ given
above.  Also notice that the degree of $L_{\alpha}(tP)$ coincides with the degree
of genericity of a tangent cone lying in the equivalence class $\alpha$.

We shall demonstrate a couple ways to
determine $L_{\alpha}(tP)$ generally.  The first is based on counting unit cubes
in $P$ while keeping track of the floral types of their tangent cones.  
The second method is the formula 
appearing in Theorem \ref{mainresult}, which uses
a set of ``local polynomials'' $h(e_\alpha)$ and $h^{-1}(e_\alpha)$ lying in a
certain associative algebra which we introduce below.  The first method is trivial,
as we are still counting lattice points in more or
less the same way.  However, this method is significant as it will facilitate the establishment
of second method.

\begin{figure} 
\centering 
\includegraphics[width=0.5\textwidth]{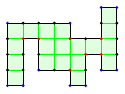}
\caption{Illustration of the running example.}
\label{orthogoncolor}
\end{figure}

\begin{figure}
$$\begin{array}{c|ccc}
\alpha & L_{\alpha,0} & L_{\alpha,1} & L_{\alpha,2}\\
\hline
\ar & 1 & -17 & 19 \\
\arx & -18 & 34 &   \\
\arxx & 11 & & \\
\arxxn & 7 & & \\
\end{array}$$
\caption{Coefficients $L_{\alpha,k}$ of the lattice point enumerator.}
\label{lpenumerator}
\end{figure}

\section{Counting cubes}

Suppose $P\subset\mathbb{R}^d$ is an integral generic orthotope.
For each floral arrangement congruence class $\alpha$ and for each integer $k\geq 0$,
let $C_{\alpha,k}(P)$ denote the number of relatively open $k$-dimensional integral unit cubes
$C\subset P$ such that the tangent cone along each point of $C$ lies in $\alpha$.
Figure \ref{localcounts} shows the values $C_{\alpha,k}(P)$ for the running example.
For each positive integer $t$, define
$$C_{\alpha}(P,t):=\sum_k C_{\alpha,k}(P)t^k.$$

\begin{figure}
$$\begin{array}{c|ccc}
\alpha & C_{\alpha,0} & C_{\alpha,1} & C_{\alpha,2}\\
\hline
\ar & 3 & 21 & 19 \\
\arx & 16 & 34 &   \\
\arxx & 11 & & \\
\arxxn & 7 & & \\
\end{array}$$
\caption{Cube counts for the example.}
\label{localcounts}
\end{figure}

Although $C_\alpha(P,t)$ does not count cubes of type $\alpha$ in the dilate
$tP$, it is closely related to the function $L_{\alpha}(tP)$:

\begin{Prp}
Suppose $P$ is an integral generic orthotope and $t$ is a positive integer.  Then
\begin{equation}
L_{\alpha}(tP)=C_{\alpha}(P,t-1).
\label{shiftrelation}
\end{equation}
\end{Prp}

\begin{proof}
Suppose $C\subset P$
is a relatively open integral unit cube of dimension $k$ and $t$ is a positive integer. 
Then there are precisely $(t-1)^k$ lattice points $x\in tC$ such that the tangent cone
at $x$ coincides with the tangent cone at each point of $C$.
Using the partition of $P$ into relatively
open integral unit cubes (of varying dimensions), the relation (\ref{shiftrelation}) follows.
\end{proof}

Figure \ref{dilation} displays the result of dilating the running example
by the factor $t=3$.  Notice that if $C\subset P$ is a relatively open integral unit square, for example, then
the interior of $3C$ contains precisely $(3-1)^2=4$ lattice points.

Using the binomial theorem,
we obtain a relation between the coefficients $L_{\alpha,k}(P)$ and
the counts $C_{\alpha,k}(P)$:

\begin{Cor}  Suppose $P$ is an integral generic orthotope, $\alpha$ is a floral arrangement,
and $k\in\{0,1,2,...,\deg\alpha\}$.  Then
$$L_{\alpha,k}(P)=\sum_{j=k}^{\deg\alpha}(-1)^{j+k}\binom{j}{k}C_{\alpha,j}(P).$$
\label{clrelation}
\end{Cor}

\begin{figure} 
\centering 
\includegraphics[width=\textwidth]{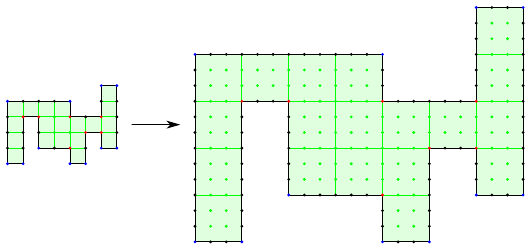}
\caption{Dilating the example.}
\label{dilation}
\end{figure}

For the running example, we compute
$$L_{\ar}(tP)=C_{\arx}(P,t-1)=3+21(t-1)+19(t-1)^2=1-17t+19t^2$$
and
$$L_{\arx}(tP)=C_{\arx}(P,(t-1))=16+34(t-1)=-18+34t,$$
yielding the first two rows of the table in Figure \ref{lpenumerator}.

\section{Local polynomials}

This section introduces and studies a set of ``local polynomials''
$h(e_\alpha)$, defined for each floral vertex congruence class $\alpha$.  
These are instrumental in
the relation, stated as Theorem \ref{mainresult}, between the number 
of lattice points and the numbers
of relatively open integral unit cubes of 
various floral types in an integral generic orthotope.

Each local polynomial has an expression
$$h(e_\alpha):=\sum_{\beta}m_{\alpha,\beta}e_\beta,$$
summing over floral vertex congruence classes $\beta$.
Here, $m_{\alpha,\beta}$ is an integer that measures the occurrence of
floral vertices congruent to $\beta$ in $\alpha$.
As such, their definition requires
careful analysis of faces and cross-sections of floral vertices, thus explaining the length
of this section.

We call the expressions $h(e_\alpha)$ ``polynomials'' because they lie 
in an infinite-dimensional, commutative, unital associative algebra $\mathcal{F}$ which 
behaves much like a ring of polynomials.  
(We note that the algebra
$\mathcal{F}$, while containing a ring of polynomials in one variable as a subring,
has two different ways to multiply.  Thus, although $\mathcal{F}$ qualifies as
having the structure of
an associative algebra, this terminology neglects 
this additional structure.)
We call the expressions $h(e_\alpha)$ ``local'' because they contain information about
incidences of floral arrangements in a given floral vertex.

A basis for $\mathcal{F}$ as a vector space consists of all expressions
$e_\alpha$, where $\alpha$ is a floral vertex congruence class,
together with $e_{\ar}$.
The multiplication in $\mathcal{F}$ is given by the rule
$$e_\alpha\cdot e_\beta:=e_{\alpha\wedge\beta},$$
and the element $e_{\ar}$ serves as the multiplicative identity element.
Notice that there is a unique monomorphism 
$\mathbb{Q}[t]\hookrightarrow\mathcal{F}$ which maps $1\mapsto e_{\ar}$ and
$t\mapsto e_{\arx}$.
(We assume throughout that the field of scalars is $\mathbb{Q}$.)

The definition of the coefficients $m_{\alpha,\beta}$ is facilitated by the observation,
which was established in \cite{richter_genericorthotopes}:

\begin{Prp}
Suppose $\alpha$ is a floral vertex and $x\in\alpha$.
Then there is a floral vertex $\beta$ such that the tangent cone at $x$
is congruent to $\beta\times\mathbb{R}^{d-\dim\beta}$.
\end{Prp}

With this, say
that $x\in\alpha$ {\it has floral type $\beta$} if the tangent cone at $x$ 
is congruent to $\beta\times\mathbb{R}^{d-k}$, where $\beta$ is a $k$-dimensional
floral vertex.
If $x$ lies in the interior of $\alpha$, then we decree that $x$ has floral type $\ar$.

Recalling \cite{richter_genericorthotopes}, we may assert:

\begin{Prp}
Suppose $\alpha$ is a floral vertex and $x\in\alpha$.
(a) If $x$ lies in the relative interior of an edge $e$ of $\alpha$,
then the floral type of $x$ coincides with the cross-section of $\alpha$ across $e$.
(b) If $x$ lies in the relative interior of a facet of $\alpha$, then $x$ has floral type $\arx$.
(c) The origin $x=\mathbf{0}$ has floral type $\alpha$.
\end{Prp}

Suppose $\alpha$ and $\beta$ are floral vertices with $\dim\alpha=d$ and
$\dim\beta=k$.  From the results above, we may define 
$$m_{\alpha,\beta}:=\sum_f \mu_{d-k}(f),$$
summing over all $(d-k)$-dimensional faces $f$ which have floral type
$\beta$, and $\mu_{d-k}$ is the $(d-k)$-dimensional orthant-counting function 
(defined in \cite{richter_genericorthotopes}).  In other words, $m_{\alpha,\beta}$ is the number of 
$(d-k)$-dimensional generalized orthants $\Omega\subset\alpha$
such that each point of $\Omega$ has floral type $\beta$.
Apparently we have $m_{\alpha',\beta'}=m_{\alpha,\beta}$ whenever $\alpha'$
is congruent to $\alpha$ or $\beta'$ is congruent to $\beta$.  Thus, we use
$m_{\alpha,\beta}$ to denote this common value when $\alpha$ and $\beta$
are floral vertex congruence classes.

The table in Figure \ref{volumepolys} displays the polynomials $h(e_\alpha)$ for $\dim\alpha\leq 3$.
Inspecting Figure \ref{good3dvertices}, for example, we may see
$$m_{\arxxnxn,\arx}=5,$$
as the floral vertex defined by $\arxxnxn$ has precisely 5 two-dimensional generalized orthants
with floral type $\beta=\arx$. 
The polynomials $h(e_\alpha)$ when $\dim\alpha=4$
appear in a table in the appendix.

\begin{figure} 
\centering 
\includegraphics[width=0.75\textwidth]{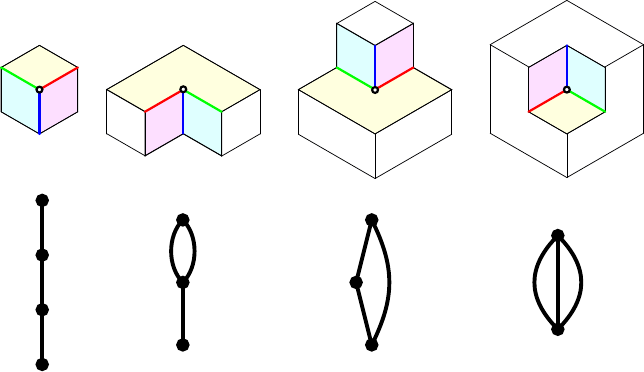}
\caption{Three-dimensional floral vertices.}
\label{good3dvertices}
\end{figure}

\begin{figure}
$$\begin{array}{c|c}
\alpha & h(e_\alpha) \\
\hline
\ar & (\ar) \\
\arx & (\ar)+(\arx) \\
\arxx & (\ar)+2(\arx)+(\arxx) \\
\arxxn & 3(\ar)+2(\arx)+(\arxxn) \\
\arxxx & (\ar)+3(\arx)+3(\arxx)+(\arxxx) \\
\arxxnx & 3(\ar)+5(\arx)+2(\arxx)+(\arxxn)+(\arxxnx) \\
\arxxnxn & 5(\ar)+5(\arx)+(\arxx)+2(\arxxn)+(\arxxnxn) \\
\arxxxn & 7(\ar)+3(\arx)+3(\arxxn)+(\arxxxn) \\
\end{array}$$
\caption{Polynomials $h(e_\alpha)$ for $\dim\alpha\leq 3$.}
\label{volumepolys}
\end{figure}

The polynomials $h(e_\alpha)$ and the coefficients $m_{\alpha,\beta}$ enjoy several properties
which we need in order to state and prove Theorem \ref{mainresult}.
Like the results in \cite{richter_genericorthotopes}, all of these properties are established 
using elementary means.  

\begin{Prp}
Suppose $\alpha$ and $\beta$ are floral vertices.
If $\dim\beta>\dim\alpha$, then no point of $\alpha$ has floral type $\beta$.
\end{Prp}

A corollary to this is that $m_{\alpha,\beta}=0$ when $\dim\beta>\dim\alpha$,
and $m_{\alpha,\beta}=0$ for all but a finite number of $\beta$ for each $\alpha$.
By a similar token, if $\dim\alpha=\dim\beta$, then the values of $m_{\alpha,\beta}$
are quickly determined:

\begin{Prp}
Suppose $\alpha$ and $\beta$ are floral vertex congruence classes with
$\dim\alpha=\dim\beta$.  Then 
$$m_{\alpha,\beta}=
\left\{\begin{array}{ccl}
1 & \hbox{if} & \beta=\alpha, \\
0 & \hbox{if} & \beta\neq\alpha. \\
\end{array}\right.$$
\end{Prp}

Recall from \cite{richter_genericorthotopes} that $\overline{\alpha}$ denotes 
the floral vertex complementary to $\alpha$.  Since
$m_{\alpha,\ar}=\mu_d(\alpha)$ coincides with the number of $(\dim\alpha)$-dimensional
orthants occupied by $\alpha$, this yields:

\begin{Prp}
If $\alpha$ is a floral vertex congruence class with $\alpha\neq\ar$, then
$m_{\alpha,\ar}+m_{\overline{\alpha},\ar}=2^{\dim\alpha}$.
\end{Prp}

Similarly, that $\partial\overline{\alpha}=\partial\alpha=\alpha\cap\overline{\alpha}$ for all $\alpha$ gives us:

\begin{Prp}
Suppose $\alpha\subset\mathbb{R}^d$ is a floral vertex and $x\in\partial\alpha$.
Then $x$ has floral type $\beta$ in $\alpha$ if and only if $x$ has floral type
$\overline{\beta}$ in $\overline{\alpha}$. 
\end{Prp}

As a corollary, we have $m_{\overline{\alpha},\overline{\beta}}=m_{\alpha,\beta}$ whenever
$\beta\neq\ar$.

Recall from \cite{richter_genericorthotopes} that if $\alpha$ and $\beta$
are floral vertices, then $\alpha\times\beta$ is a floral vertex
given by the conjunction of the series-parallel diagrams underlying
$\alpha$ and $\beta$.  This in turn implies that
$$h(e_{\alpha}\cdot e_{\beta})= h(e_{\alpha\wedge\beta})=h(e_\alpha)\cdot h(e_\beta)$$
for all floral vertex congruence classes $\alpha$ and $\beta$.
Since $h(e_{\ar})=e_{\ar}$ we may say:

\begin{Prp}
The map $h:\mathcal{F}\rightarrow\mathcal{F}$ is an algebra automorphism.
\end{Prp}

We illustrate how the properties outlined above serve to compute small examples
of the expressions $h(e_\alpha)$.  For example, we may write:
\begin{align*}
h(\arxxuxxn) &= h(\arxxn)\cdot h(\arxxn) \\
&=\left[3(\ar)+2(\arx)+(\arxxn) \right]\cdot\left[3(\ar)+2(\arx)+(\arxxn) \right] \\
& = 9(\ar)\cdot(\ar)+ 4(\arx)\cdot(\arx)+(\arxxn)\cdot(\arxxn) \\
& +12(\ar)\cdot(\arx)+6(\ar)\cdot(\arxxn)+4(\arx)\cdot(\arxxn) \\
& = 9(\ar)+12(\arx)+4(\arxx)+6(\arxxn)+4(\arxxnx)+(\arxxuxxn),
\end{align*}
as $h$ is an algebra automorphism. 
Using this and the facts $m_{\overline{\alpha},\overline{\beta}}=m_{\alpha,\beta}$
and $m_{\alpha,\ar}+m_{\overline{\alpha},\ar}=2^{\dim \alpha}$
we also obtain:
\begin{align*}
h(\arxxuxx) &= (16-9)(\ar)+12\overline{(\arx)}+4\overline{(\arxx)}+6\overline{(\arxxn)}+4\overline{(\arxxnx)}+\overline{(\arxxuxxn)} \\
& = 7(\ar)+12(\arx)+4(\arxxn)+6(\arxx)+4(\arxxnxn)+(\arxxuxx).
\end{align*}
Notice that both of these appear in a table in the appendix.

\subsection{The polynomials $h^{-1}(e_\alpha)$}

For each non-negative integer $d$,
let $\mathcal{F}_d$ denote the subspace of $\mathcal{F}$ spanned by those
$e_\alpha$ for which $\dim\alpha\leq d$.  The dimension of 
$\mathcal{F}_d$ is 
$$\dim\mathcal{F}_d=\sum_{k=0}^d A_k=1+1+2+4+\cdots+A_d,$$
where $A_k$ denotes the number of series-parallel diagrams with precisely $k$
segments.  As we observed in \cite{richter_genericorthotopes}, $A_k$ is the
$k$th term in sequence A000084 in the Online
Encyclopedia of Integer Sequences.  Several values of $A_d$ and
$\sum_{k=0}^d A_k$ appear in Figure \ref{oeis}.

\begin{figure}
$$\begin{array}{c|c|c}
 &  & \dim\mathcal{F}_d \\
d & A_d & =\sum_{k=0}^d A_k \\
\hline
0 & 1 & 1 \\
1 & 1 & 2 \\
2 & 2 & 4 \\
3 & 4 & 8 \\
4 & 10 & 18 \\
5 & 24 & 42 \\
6 & 66 & 108 \\
7 & 180 & 288 \\
8 & 522 & 810 \\
9 & 1532 & 2342 \\
10 & 4624 & 6966 \\
\end{array}$$
\caption{OEIS sequence A000084 and partial sums.}
\label{oeis}
\end{figure}

% copied from the OEIS data on sequence A000084:
% 1, 2, 4, 10, 24, 66, 180, 522, 1532, 4624, 14136,

From the results above, we know that $h^{-1}$ is an automorphism of the algebra
$\mathcal{F}$.  Morever, the restriction of $h-\mathrm{Id}$ to 
the subspace $\mathcal{F}_d$ is nilpotent, whence we have
$$h^{-1}=\sum_{k=0}^{d}(-1)^k(h-\mathrm{Id})^k.$$
We tabulate $h^{-1}(e_\alpha)$ for $\dim\alpha\leq 3$ in
Figure \ref{eulerpolys3} and those with $\dim\alpha=4$ in the appendix.
We shall exhibit more properties of $h^{-1}(e_\alpha)$ after we establish the main result.

\begin{figure}
$$\begin{array}{c|c}
\alpha & h^{-1}(e_\alpha) \\
\hline
\ar & (\ar) \\
\arx & -(\ar)+(\arx) \\
\arxx & (\ar)-2(\arx)+(\arxx) \\
\arxxn & -(\ar)-2(\arx)+(\arxxn) \\
\arxxx & -(\ar)+3(\arx)-3(\arxx)+(\arxxx) \\
\arxxnx & (\ar)+(\arx)-2(\arxx)-(\arxxn)+(\arxxnx) \\
\arxxnxn & (\ar)+(\arx)-(\arxx)-2(\arxxn)+(\arxxnxn) \\
\arxxxn & -(\ar)+3(\arx)-3(\arxxn)+(\arxxxn) \\
\end{array}$$
\caption{Polynomials $h^{-1}(e_\alpha)$ for $\dim\alpha\leq 3$.}
\label{eulerpolys3}
\end{figure}

\section{The Main Result}

We use the expressions $h^{-1}(e_\alpha)$ to state our main result.  
For notational convenience, define 
$D\in\mathrm{End}(\mathcal{F})$ by
$$D(e_\alpha):=2^{\dim{\alpha}}e_\alpha$$
for each $\alpha$.
With this, we assert:

\begin{Thm}
Suppose $P$ is an integral generic orthotope and $k\in\{0,1,2,...,d\}$.  Then
\begin{equation}
\sum_{\alpha}L_{\alpha,k}(P)e_\alpha=2^{k-d}\sum_{\deg\alpha=k}C_{\alpha,\deg\alpha}(P)D(h^{-1}(e_\alpha)).
\end{equation}
\label{mainresult}
\end{Thm}

Before proving it, 
we demonstrate this formula using our running example.  
Using the formula from Theorem \ref{mainresult} 
with the values tabulated in Figure \ref{localcounts} and $k=0$, we have
\begin{align*}
& 2^{k-d}\sum_{\deg\alpha=k}C_{\alpha,\deg\alpha}(P)D(h^{-1}(e_\alpha)) \\
&= 2^{0-2}\sum_{\deg\alpha=0}C_{\alpha,0}(P)D(h^{-1}(e_\alpha)) \\
&= 2^{-2}\left[C_{\arxx,0}(P) D(h^{-1}(\arxx))+C_{\arxxn,0}(P)D(h^{-1}(\arxxn))\right] \\
&= \frac{1}{4}\left[11D(h^{-1}(\arxx))+7D(h^{-1}(\arxxn))\right] \\
&= \frac{11}{4}D\left((\ar)-2(\arx)+(\arxx)\right)
+\frac{7}{4}D\left(-(\ar)-2(\arx)+(\arxxn)\right) \\
&= \frac{11}{4}\left((\ar)-4(\arx)+4(\arxx)\right)
+\frac{7}{4}\left(-(\ar)-4(\arx)+4(\arxxn)\right) \\
&= 1(\ar)-18(\arx)+11(\arxx)+7(\arxxn) \\
&= \sum_{\alpha}L_{\alpha,0}(P)e_\alpha,
\end{align*}
yielding the first column of the table in Figure \ref{lpenumerator}.
Likewise, with $k=1$,
\begin{align*}
2^{-1}C_{\arx,1}(P)D(h^{-1}(\arx)) = & \frac{1}{2}\cdot 34\cdot D(-(\ar)+(\arx)) \\
=& \frac{1}{2}\cdot 34\left[-(\ar)+2(\arx)\right] \\
=& -17(\ar)+34(\arx) \\
=& L_{\ar,1}(P)e_{\ar}+L_{\arx,1}(P)e_{\arx}. 
\end{align*}

We devote the remainder of this section to proving Theorem \ref{mainresult}.
We accomplish this by first computing $L_{\alpha,0}(P)$ for an integral generic orthotope
and then considering cross-sections to obtain the general formula.

Define the {\it Euler vector} of $P$ by
$$\phi(P):=\sum_{\alpha}L_{\alpha,0}(P)e_\alpha\in\mathcal{F},$$
summing over all floral arrangement congruence classes $\alpha$ with $\deg\alpha\leq d$.
The Euler vector for our running example is thus
$$\phi(P)=e_{\ar}-18e_{\arx}+11e_{\arxx}+7e_{\arxxn}.$$

\begin{Lem}
Suppose $P$ is an integral generic orthotope, $\beta$ is a 
floral vertex congruence class, and $k\in\{0,1,2,...,d\}$.  Then
\begin{equation}
\sum_{\alpha}2^{-\dim\alpha}m_{\alpha,\beta}C_{\alpha,k}(P)
=2^{-\dim\beta} \binom{\deg\beta}{k}C_{\beta,\deg\beta}(P),
\end{equation}
summing over all floral vertex congruence classes $\alpha$ with $\dim\alpha\leq d$.
\label{countinglemma}
\end{Lem}

\begin{proof}
This is a double-counting formula.
Notice that $C_{\beta,\deg\beta}(P)$ is the number of $(\deg\beta)$-dimensional relatively
open integral unit cubes of type $\beta$.  For each $k$-dimensional relatively
open unit cube $C$ of type $\alpha$, the number of $(\deg\beta)$-dimensional relatively open integral unit cubes
of type $\beta$ incident to $C$ coincides with the number
$m_{\alpha,\beta}$ of orthants of floral type $\beta$ contained in the floral vertex $\alpha$
(abusing notation slightly).
The auxiliary values $2^{-\dim\alpha}$, $2^{-\dim\beta}$, and $\binom{\deg\beta}{k}$
are normalizing factors.  In particular, notice that $\binom{\deg\beta}{k}$
counts the number of orthographic subspaces of dimension $k$ in $\mathbb{R}^{\deg\beta}$.
\end{proof}

We demonstrate the formula above with our running example, using
$\beta=\arx$ and $k=0$.  We have
\begin{align*}
&\sum_{\alpha}2^{-\dim\alpha}m_{\alpha,\beta}C_{\alpha,k}(P) \\
&= \sum_{\alpha}2^{-\dim\alpha}m_{\alpha,\arx}C_{\alpha,0}(P) \\
&= 2^{-\dim(\ar)}m_{\ar,\arx}C_{\ar,0}(P)+2^{-\dim(\arx)}m_{\arx,\arx}C_{\arx,0}(P) \\
& +2^{-\dim(\arxx)}m_{\arxx,\arx}C_{\arxx,0}(P)+2^{-\dim(\arxxn)}m_{\arxxn,\arx}C_{\arxxn,0}(P) \\
&= 2^{-0}\cdot 0\cdot 3+2^{-1}\cdot 1\cdot 16+2^{-2}\cdot 2\cdot 11+2^{-2}\cdot 2\cdot 7 
= 17. 
\end{align*}
On the other hand,
\begin{align*}
&\binom{\deg\beta}{k}2^{-\dim\beta} C_{\beta,\deg\beta}(P) \\
&=\binom{\deg(\arx)}{0}2^{-\dim(\arx)} C_{\arx,\deg(\arx)}(P) \\
&=\binom{1}{0}2^{-1} C_{\arx,1}(P) 
=1\cdot\frac{1}{2}\cdot 34 =17.
\end{align*}

The formula in Lemma \ref{countinglemma} is a generalization of a formula from \cite{richter_genericorthotopes}
for the volume expressed
with the orthant-counting function $\mu_d$.  Thus, with $\beta=\ar$ and $k=0$, we have 
$$\sum_{\alpha}2^{-\dim\alpha}m_{\alpha,\beta}C_{\alpha,k}(P)=\sum_{\alpha}2^{-\dim\alpha}m_{\alpha,\ar}C_{\alpha,0}(P)
=2^{-d}\sum_\alpha\mu_d(\alpha)n_\alpha(P),$$
where $2^{-\dim\alpha}m_{\alpha,\ar}=2^{-d}\mu_d(\alpha)$ is the fraction of all $d$-dimensional
orthants occupied by $\alpha\in\mathbb{R}^d$ and $C_{\alpha,0}=n_\alpha(P)$ is the number of lattice
points in $P$ having floral type $\alpha$.  On the other hand, we have
$$2^{-\dim\beta} \binom{\deg\beta}{k}C_{\beta,\deg\beta}(P)=2^{-0} \binom{d}{0}C_{\ar,d}(P)=\mathrm{Vol}_d(P).$$

\begin{Prp}
If $P$ is an integral generic orthotope, then
\begin{equation}
\sum_{k=0}^d(-1)^k\sum_{\alpha}2^{-\dim\alpha}C_{\alpha,k}(P)e_\alpha
=2^{-d}\sum_{\deg\beta=0}C_{\beta,0}(P)h^{-1}(e_\beta).
\end{equation}
\label{localcount}
\end{Prp}

We remark that the stipulation $\deg\beta=0$ in this proposition is equivalent to requiring
that $\beta$ be a floral vertex.  In other words, the domain of summation in the right-hand
side of this formula coincides with $d$-dimensional floral vertex congruence classes.

\begin{proof}
This is a straightforward computation using the results above.  In detail, notice
\begin{align*}
&\sum_{k=0}^d(-1)^k\sum_{\alpha}2^{-\dim\alpha}C_{\alpha,k}(P)h(e_\alpha) \\
&=\sum_{k=0}^d (-1)^k\sum_{\alpha}2^{-\dim\alpha}C_{\alpha,k}(P)\sum_{\beta} m_{\alpha,\beta}e_\beta 
\hbox{ (definition of $h$)} \\
&=\sum_{k=0}^d (-1)^k\sum_{\beta}\sum_{\alpha} 2^{-\dim\alpha}m_{\alpha,\beta}C_{\alpha,k}(P)e_\beta  \\
&=\sum_{k=0}^d (-1)^k\sum_{\beta}\binom{\deg\beta}{k}2^{-\dim\beta}C_{\beta,\deg\beta}(P)e_\beta 
\hbox{ (from Proposition \ref{localcount})} \\
&=\sum_{\beta}2^{-\dim\beta}C_{\beta,\deg\beta}\left[\sum_{k=0}^{\deg\beta} (-1)^k \binom{\deg\beta}{k}\right]e_\beta. 
\end{align*}
The expression in square brackets vanishes when $\deg\beta>0$ and equals 1 when $\deg\beta=0$, so we obtain
\begin{align*}
&\sum_{k=0}^d(-1)^k\sum_{\alpha}2^{-\dim\alpha}C_{\alpha,k}(P)h(e_\alpha) \\
&=\sum_{\deg\beta=0}2^{-\dim\beta}C_{\beta,\deg\beta}(P)e_\beta \\
&=2^{-d}\sum_{\deg\beta=0}C_{\beta,0}(P)e_\beta
\end{align*}
Applying $h^{-1}$, the result follows.
\end{proof}

\begin{Cor}
Suppose $P$ is an integral generic orthotope.  Then
$$\phi(P)=2^{-d}\sum_{\deg\alpha=0}C_{\alpha,0}(P)D(h^{-1}(e_{\alpha})).$$
\label{eulervector}
\end{Cor}

\begin{proof}
Notice
\begin{align*}
\phi(P) &= \sum_{\alpha}L_{\alpha,0}(P)e_\alpha 
\hbox{ (definition of $\phi$)} \\
&= \sum_{\alpha}\sum_{k=0}^{d}(-1)^k C_{\alpha,k}(P)e_\alpha 
\hbox{ (from Corollary \ref{clrelation})} \\
&= \sum_{k=0}^{d}(-1)^k\sum_{\alpha} C_{\alpha,k}(P)e_\alpha \\
&= \sum_{k=0}^{d}(-1)^k\sum_{\alpha} 2^{-\dim\alpha}C_{\alpha,k}(P)D(e_\alpha)
\hbox{ (definition of $D$)} \\
&=2^{-d}\sum_{\deg\alpha=0}C_{\alpha,0}(P)D(h^{-1}(e_\alpha)) 
\hbox{ (from Proposition \ref{localcount})} \\
\end{align*}
\end{proof}

With this, we prove Theorem \ref{mainresult} by considering cross-sections.  Suppose $k$ is a positive
integer, $I\subset[d]$ with $|I|=k$, and
$\lambda:I\rightarrow(\mathbb{Z}+\frac{1}{2})$ is a tuple of half integers.
Recall from \cite{richter_genericorthotopes} that the cross-section 
$P\cap\Pi_{I,\lambda}$ is a generic orthotope of dimension
$d-k$, where $\Pi_{I,\lambda}$ is the
generalized hyperplane determined by $(I,\lambda)$.
Thus, for an appropriate choice of $v\in\mathbb{R}^d$ (depending on $\lambda$), the
translation $v+\left(P\cap\Pi_{I,\lambda}\right)$ is integral and congruent to $P\cap\Pi_{I,\lambda}$.
Counting integral unit cubes of dimension $k$ in $P$ thus amounts to counting 
lattice points in $P\cap\Pi_{I,\lambda}$.
This then yields, for any floral arrangement $\alpha$,
$$C_{\alpha,k}(P)=\sum_{(I,\lambda)} C_{\alpha,0}(v_\lambda+P\cap\Pi_{I,\lambda}),$$
summing over all choices of $(I,\lambda)$ with $|I|=k$.  
Theorem \ref{mainresult} then follows immediately.
\hspace{\fill}$\square$

\section{Auxiliary results}

This section has several results which supplement the preceding theory.

\subsection{Euler vector}

The following helps to explain why we refer to $\phi(P)$ as the Euler vector:

\begin{Prp}
If $P$ is an integral generic orthotope of dimension $d$ and Euler characteristic
$\chi(P)$, then
$$\chi(P)=(-1)^d L_{\ar,0}(P).$$
\end{Prp}

\begin{proof}
For a $d$-dimensional integral generic orthotope, define
$$f_d(P):=\sum_{k=0}^d(-1)^k C_{\ar,k}(P).$$
We will see that $(-1)^d f_d$ agrees with the Euler characteristic
valuation.  
The proof follows by a routine induction argument on $d$.  

Suppose first that $d=1$.
Then $P$ is a disjoint union of closed intervals and $-f_d(P)$ is the number of such intervals of $P$, i.e.\
the Euler characteristic of $P$.

Next, suppose $P$ is an integral generic orthotope of dimension $d$ and $\Pi$ is a 
$(d-1)$-dimensional integral hyperplane such that 
the half-spaces
on either side of $\Pi$ intersect $P$ in $d$-dimensional orthogonal polytopes, say $P^+$ and $P^-$.
Thus, if $H^\pm$ are the open half-spaces on either side of $\Pi$, then $P^+$ (resp.\ $P^-$) 
is the closure of the intersection $P\cap H^+$ (resp.\ $P\cap H^-$).
We observe that
$$C_{\ar,k}\left(P^+\cup P^-\right)=C_{\ar,k}\left(P^+\right)+C_{\ar,k}\left(P^-\right)+C_{\ar,k}\left(P^+\cap P^-\right)$$
for all $k$, although we must interpret this equation properly.  
Thus, while $C_{\ar,k}(P)$ generally denotes the number of relatively open $k$-dimensional integral unit cubes
of floral type $\ar$ in $P$, we notice that each such cube necessarily lies in the relative interior of $P$.
In particular, we notice that the orthogonal
polytopes $P=P^+\cup P^-$, $P^+$, and $P^-$ are all $d$-dimensional generic orthotopes,
whereas $P^+\cap P^-$ is a generic orthotope of dimension $d-1$.  Thus, for example, if 
$x$ lies in the relative interior of 
$P^+\cap P^-$, then $x\in \partial P^+$ and $x\in\partial P^-$ and $x$ cannot have floral
type $\ar$ in $P^+$ or in $P^-$.
From this observation, we see that
$$f_d(P^+\cup P^-)=f_d(P^+)+f_d(P^-)+f_{d-1}(P^+\cap P^-).$$
Thus,
\begin{align*}
&(-1)^df_d(P^+\cup P^-)\ \\
&=(-1)^df_d(P^+)+(-1)^df_d(P^-)-(-1)^{d-1}f_{d-1}(P^+\cap P^-),
\end{align*}
Notice that $C_{\ar,k}(I^d)$ is either 0 or 1 depending respectively on whether
$k<d$ or $k=d$, where $I^d$ is the standard unit cube.
In particular, we have $f_d(I^d)=(-1)^d$ for all $d$.  
Therefore $(-1)^df_d=\chi$ for all $d$.
\end{proof}

This in turn yields:

\begin{Cor}
If $P$ is a $d$-dimensional integral generic orthotope and $\alpha$ is a floral arrangement
congruence class, then $(-1)^{\deg\alpha}L_{\alpha,0}(P)$ is the
sum of the Euler characteristics of the faces of $P$ whose tangent cones
lie in $\alpha$.
\end{Cor}

Thus, the Euler vector $\phi(P)=\sum_{\alpha}L_{\alpha,0}(P)e_\alpha$ is a tabulation of these
values ranging over all floral arrangement congruence classes.
This is apparent in the running example.
Thus, notice that each vertex has Euler characteristic $1$, while
$P$ has precisely $L_{\arxx,0}(P)$ and $L_{\arxxn,0}(P)$ vertices
of types $\arxx$ and $\arxxn$ respectively.
Similarly, $P$ has $-L_{\arx,0}=18$ edges, with each having Euler characteristic
$1$.
Finally, $L_{\ar,0}(P)=1$, as $P$ is homeomorphic to the disc $I^2$.

\subsection{Ehrhart-Macdonald reciprocity}

The preceding result quickly leads to a derivation of Ehrhart-Macdonald 
reciprocity for integral generic orthotopes.
Recall that $tP$ contains precisely $L_{\ar}(tP)$ interior lattice points.
If $t>0$, then let $L(-tP)$ denote the evaluation of the polynomial $L(tP)$ at $-t$.

\begin{Prp}
If $P$ is a $d$-dimensional integral generic orthotope, then
$$L_{\ar}(tP)=(-1)^d L(-tP).$$
\end{Prp}

\begin{proof}
The Euler characteristic of $P$ is given by
\begin{align*}
\chi(P) &=\sum_{\alpha}\sum_k(-1)^k C_{\alpha,k}(P) \\
&=\sum_{\alpha}L_{\alpha,0}(P) \\
&=L_{\ar,0}(P)+\sum_{\alpha\neq\ar}L_{\alpha,0}(P), 
\end{align*}
where $\alpha$ denotes a floral arrangement congruence class.
Using $\chi(P)=(-1)^d L_{\ar,0}(P)$ from above, this yields
$$\sum_{\alpha\neq\ar}L_{\alpha,0}(P)=\left[(-1)^d-1\right]L_{\ar,0}(P).$$
As in the proof of Theorem \ref{mainresult}, we consider cross-sections of dimension
$d-k$ to obtain
$$\sum_{\alpha\neq\ar}L_{\alpha,k}(P)=\left[(-1)^{d-k}-1\right]L_{\ar,k}(P)$$
for all $k$.
The result then follows immediately.
\end{proof}

In the running example, we notice
$$L(tP)=L_{\ar}(tP)+L_{\arx}(tP)+L_{\arxx}(tP)+L_{\arxxn}(tP)=1+17t+19t^2,$$
while
$$L_{\ar}(tP)=1-17t+19t^2.$$

\subsection{Properties of $h^{-1}(e_\alpha)$}

For any $\alpha$, let $s_{\alpha,\beta}\in\mathbb{Z}$ denote the coefficients such that
$$h^{-1}(e_\alpha)=\sum_\beta s_{\alpha,\beta}e_\beta,$$
summing over floral vertex congruence classes $\beta$.

Recall from \cite{richter_genericorthotopes} that the bouquet sign function $\sigma$
satisfies $\sigma(\alpha)=(-1)^{\rho(\alpha)}$, where $\rho(\alpha)$ is the number of loops
in the series-parallel diagram defining $\alpha$ (which coincides with the number of disjunctions
used in its read-once expression).  Using Theorem \ref{mainresult} and the fact that both
of $(-1)^d L_{\alpha,0}(P)$ and the sum of the bouquet signs of the vertices of $P$ are
equal to the Euler characteristic of $P$ yields:

\begin{Prp}
$s_{\alpha,\ar}=(-1)^{\dim\alpha}\sigma(\alpha)$ for every floral vertex congruence class.
\end{Prp}

We summarize several other properties of the polynomials $h^{-1}(e_\alpha)$ and the coefficients $s_{\alpha,\beta}$:

\begin{itemize}

\item
$s_{\alpha,\alpha}=1$ for all $\alpha$.

\item
$s_{\alpha,\beta}=0$ if $\dim\alpha=\dim\beta$ and $\alpha\neq\beta$.

\item
$s_{\alpha,\beta}=0$ whenever $\dim\beta>\dim\alpha$.

\item
$s_{\alpha,\beta}=s_{\overline{\alpha},\overline{\beta}}$ whenever $\beta\neq\ar$.

\item
$s_{\alpha,\beta}$ is the number of cross-sections of type $\beta$ in $\alpha$
when $\dim\beta=\dim\alpha-1$.

\item
$s_{\alpha,\arx}=-(-1)^{\dim\alpha}\sum_{\beta}\sigma(\beta)$, summing over all facets
$\beta$ of $\alpha$.

\item
$\sum_{\beta}2^{\dim\beta}s_{\alpha,\beta}=\sigma(\alpha)$, summing over all floral vertex congruence
classes $\beta$.

\end{itemize}

One may establish all of these in a manner similar to
our study of the polynomials $h(e_\alpha)$.
Notice that the last of these is a reflection of Ehrhart-Macdonald reciprocity.
All of these properties are apparent in the tables in Figures
\ref{eulerpolys3} and in the appendix.

\subsection{Special Cases}

Let $P\subset\mathbb{R}^d$ be a $d$-dimensional integral generic orthotope.
We have already seen that
$(-1)^dL_{\ar,0}(P)$ coincides with the Euler characteristic of $P$.  The purpose of this
section is to detail some other special formulas involving some familiar valuations
of $P$, if only to better understand the notation.  In particular, we notice:

\begin{itemize}

\item
$\sum_{\deg\alpha=0} L_{\alpha,0}(P)=\#(\hbox{vertices of }P).$

\item
$\sum_{\deg\alpha=1} L_{\alpha,0}(P)=-\#(\hbox{edges of }P).$

\item
$L_{\arx,d-1}(P)=-2L_{\ar,d-1}(P)=\mathrm{Volume}_{d-1}(\partial P).$

\item
$L_{\ar,d}(P)=C_{\ar,d}(P)=\mathrm{Volume}_d(P).$

\end{itemize}

These are all trivial given an understanding of the notation $L_{\alpha,k}(P)$.
Notice the third of these is a statement of the well-known fact that the degree-$(d-1)$
coefficient of the Ehrhart polynomial of a $d$-dimensional polytope $P$ is half
of the negative of the measure of $\partial P$.

\section{A multivariable generalization}

The enumeration of lattice points described above can be adapted to
the case where coordinates are dilated independently.
Thus, if $\mathbf{t}=(t_1,t_2,...,t_d)$ is a tuple of positive integers, then
define the {\it generalized Ehrhart function} of $P$ by
$$L(\mathbf{t}P):=\#(\mathbf{t}P\cap\mathbb{Z}^d),$$
where $\mathbf{t}P$ denotes the image of $P$ under the dilation map
$$(x_1,x_2,...,x_d)\mapsto(t_1x_1,t_2x_2,...,t_dx_d).$$

All of the results in this section follow through an analysis identical to
that used above.  As above, we write
$$L(\mathbf{t}P):=\sum_\alpha L_\alpha(\mathbf{t}P),$$
where $L_\alpha(\mathbf{t}P)$ denotes the number of lattice points in $\mathbf{t}P$ at which
the tangent cone lies in $\alpha$.

\begin{Prp}
Suppose $P$ is an integral generic orthotope
and $\alpha$ is a floral arrangement.  
Then $\mathbf{t}\mapsto L_\alpha(\mathbf{t}P)$ is a polynomial function of $\mathbf{t}=(t_1,t_2,...,t_d)$
of degree $\deg\alpha$
in which every term is a square-free 
monomial in $\{t_1,t_2,...,t_d\}$.
\end{Prp}

For $I\subset[d]$, let $\mathbf{t}_I=\prod_{i\in I}t_i$ be the square-free monomial corresponding
to $I$.  Then we write
$$L_\alpha(\mathbf{t}P):=\sum_{k=0}^{\deg\alpha}\sum_{|I|\leq k}L_{\alpha,I}(P)\mathbf{t}_I.$$
We obtain explicit expressions for $L_{\alpha,I}(P)$ as corollary of
Theorem \ref{mainresult}.
Thus, given a floral arrangement $\alpha$ and $I\subset[d]$, let $C_{\alpha,I}(P)$
denote the number of relatively open integral unit cubes of dimension $|I|$ in $P$
whose tangent cones are congruent to $\alpha$ and which are parallel to the hyperplane
$\Pi_I=\mathrm{span}\{e_i:i\in I\}$.  Then the formula gives us:
\begin{equation}
\sum_\alpha L_{\alpha,I}(P)e_\alpha=2^{|I|-d}\sum_{\deg\alpha=|I|}C_{\alpha,I}(P)D(h^{-1}(e_\alpha)).
\label{generalizedehrhart}
\end{equation}

We illustrate this for the running example.  The cube counts $C_{\alpha,I}(P)$ appear in
the table in Figure \ref{multivariablecounts2d} and the coefficients $L_{\alpha,I}(P)$
appear in the table in Figure \ref{multivariablecoefs2d}.

\begin{figure}
$$\begin{array}{c|cccc}
\alpha & C_{\alpha,\{\}} & C_{\alpha,\{1\}} & C_{\alpha,\{2\}} & C_{\alpha,\{1,2\}}\\
\hline
\ar & 3 & 12 & 9 & 19 \\
\arx & 16 & 14 & 20 &   \\
\arxx & 11 & & &  \\
\arxxn & 7 & & & \\
\end{array}$$
\caption{Multivariable cube counts for the example.}
\label{multivariablecounts2d}
\end{figure}

\begin{figure}
$$\begin{array}{c|cccc}
\alpha & L_{\alpha,\{\}} & L_{\alpha,\{1\}} & L_{\alpha,\{2\}} & L_{\alpha,\{1,2\}}\\
\hline
\ar & 1 & -7 & -10 & 19 \\
\arx & -18 & 14 & 20 &   \\
\arxx & 11 & & &  \\
\arxxn & 7 & & & \\
\end{array}$$
\caption{Multivariable cube counts for the example.}
\label{multivariablecoefs2d}
\end{figure}

\section{Example in 3 dimensions}

Define a 3-dimensional integral orthogonal polytope as $P:=\bigcup_{v\in S}\left(v+[0,1]^3\right)$,
where $S\subset\mathbb{R}^3$ appears in Figure \ref{examplecorners}.
Evidently $P$ lies in the primary octant in $\mathbb{R}^3$.
One should imagine $P$ as an assembly of several 
stacks of unit cubes resting ``skyscraper style'' on a flat surface representing the $(x,y)$-plane
in $\mathbb{R}^3$.  A sketch of $P$ appears in Figure \ref{exampleehrhart}.  
The shading in the figure is intended to illustrate the annular facet
where $P$ makes contact with the $(x,y)$ plane.  The numbers
appearing in the sketch tell the height of each stack.  Notice that $P$ is
a solid 3-dimensional torus.

\begin{figure}[h]
$$S=\left\{\begin{array}{ccccccc}
(0,0,0), & (1,0,0), & (2,0,0), & (3,0,0), & (4,0,0), & (0,1,0), & (1,1,0), \\
(2,1,0), & (3,1,0), & (0,2,0), & (2,2,0), & (0,3,0), & (1,3,0), & (2,3,0), \\
(1,0,1), & (2,0,1), & (3,0,1), & (4,0,1), & (2,1,1), & (3,1,1), & (2,2,1), \\
(0,3,1), & (1,3,1), & (2,3,1), & (1,0,2), & (2,0,2), & (3,0,2), & (4,0,2) \\
\end{array}\right\}.$$
\caption{Generating corners of the example.}
\label{examplecorners}
\end{figure}

\begin{figure}[h] 
\centering 
\includegraphics[width=0.333\textwidth]{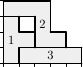}
\caption{A 3-dimensional integral generic orthotope.}
\label{exampleehrhart}
\end{figure}

Figure \ref{combinatorics} displays the values of $C_{\alpha,k}(P)$ and $L_{\alpha,k}(P)$.
Using the formula from Theorem \ref{mainresult} 
with the tabulated values, for example, one may compute the Euler vector:
\begin{align*}
\phi(P)= 
& 2^{-3}\left[15D(h^{-1}(\arxxx))+11D(h^{-1}(\arxxnx))\right. \\
 &\left.+5D(h^{-1}(\arxxnxn))+1D(h^{-1}(\arxxxn))\right] \\
=& \frac{15}{8}\left[-(\ar)+6(\arx)-12(\arxx)+8(\arxxx)\right] \\
+& \frac{11}{8}\left[(\ar)+2(\arx)-8(\arxx)-4(\arxxn)+8(\arxxnx)\right] \\
+& \frac{5}{8}\left[(\ar)+2(\arx)-4(\arxx)-8(\arxxn)+8(\arxxnxn)\right] \\
+& \frac{1}{8}\left[-(\ar)+6(\arx)-12(\arxxn)+8(\arxxxn)\right] \\
=& 0(\ar)+16(\arx)-36(\arxx)-12(\arxxn) \\
 & +15(\arxxx)+11(\arxxnx)+5(\arxxnxn)+1(\arxxxn).
\end{align*}
In particular, notice that $L_{\ar,0}(P)=(-1)^3\chi(P)=0$, as $P$ is a solid 3-dimensional torus.

We may also use the formula (\ref{generalizedehrhart}) to compute the multivariable Ehrhart polynomials.  These
appear in Figure \ref{coefs}.  Reading the table, for example, one sees
$$L_{\arx}(\mathbf{t}P)=16-(32t_1+24t_2-28t_3)+(28t_1t_2+32t_1t_3+20t_2t_3)$$
yields the number of lattice points in $\mathbf{t}P$ whose tangent cones have type $\arx$.

\begin{figure}[h]
$$\begin{array}{c|cccc|cccc}
\alpha & C_{\alpha,0} & C_{\alpha,1} & C_{\alpha,2} & C_{\alpha,3} & L_{\alpha,0} & L_{\alpha,1} & L_{\alpha,2} & L_{\alpha,3} \\
\hline
\ar & 1 & 17 & 44 & 28 & 0 & 13 & -40 & 28 \\
\arx & 12 & 76 & 80 & & 16 & -84 & 80 &  \\
\arxx & 32 & 68 & & & -36 &  68 & & \\
\arxxn & 4 & 16 & & & -12 & 16 & & \\
\arxxx & 15 & & & & 15 & & & \\
\arxxnx & 11 & & & & 11 & & & \\
\arxxnxn & 5 & & & & 5 & & & \\
\arxxxn & 1 & & & & 1 & & & \\
\end{array}$$
\caption{Cube counts $C_{\alpha,k}$ and coefficients $L_{\alpha,k}$ of the example.}
\label{combinatorics}
\end{figure}

\begin{figure}[h]
$$\begin{array}{c|c|ccc|ccc|c}
\alpha & L_{\alpha,\emptyset} & L_{\alpha,1} & L_{\alpha,2} & L_{\alpha,3} & L_{\alpha,12} & L_{\alpha,13} & L_{\alpha,23}
& L_{\alpha,123} \\
\hline
\ar & 0 & 6 & 5 & 2 & -14 & -16 & -10 & 28 \\
\arx & 16 & -32 & -24 & -28 & 28 & 32 & 20  \\
\arxx & -36 & 28 & 22 & 18 \\
\arxxn & -12 & 4 & 2 & 10 \\
\arxxx &  15 \\
\arxxnx & 11 \\
\arxxnxn & 5 \\
\arxxxn & 1 \\
\end{array}$$
\caption{Coefficients $L_{\alpha,I}$ of the example.}
\label{coefs}
\end{figure}

\section{Conclusion}

We have described a theory of Ehrhart polynomials adapted to integral generic orthotopes.
We anticipate further development concerning combinatorics of generic orthotopes.
This note was concerned solely with generic orthotopes which
are integral, and one can imagine extending these results to generic orthotopes which
are rational.  Even further, we anticipate a version of the Euler-Maclaurin summation formulae
for generic orthotopes in a vein similar to that of \cite{BV_1997}.

This author is particularly curious about the space of generic orthotopes with a given
count data set $\{C_{\alpha,\deg\alpha}\}$.  As we have seen, the number of lattice points in $P$
depends only on these values.  However, even in dimension $d=2$, it is easy to find
non-congruent generic orthogons which have the same data $\{C_{\alpha,\deg\alpha}\}$
(and therefore the same number of lattice points).  Can we say anything definitive about
the space of integral generic orthotopes having the same combinatorial data?

We introduced the algebra $\mathcal{F}$ and the ``local polynomials'' 
$h(e_\alpha),h^{-1}(e_\alpha)\in\mathcal{F}$ in order to express our formula
in Theorem \ref{mainresult}.  One notices that these objects can be defined
entirely within the theory of Boolean functions.  That is, they are intrinsic
to the study of read-once Boolean functions.  The question of how these
polynomials and especially the coefficients $m_{\alpha,\beta},s_{\alpha,\beta}\in\mathbb{Z}$
may arise in studies of complexity measures of read-once Boolean functions therefore arises naturally,
cf.\ \cite{LM_2021}.

\section{Appendix}

This section displays the polynomials $h(e_\alpha)$ and $h^{-1}(e_\alpha)$ where
$\dim\alpha=4$.

\begin{landscape}
\begin{figure}[p]
$$\begin{array}{c|c}
\alpha & h(e_\alpha) \\
\hline
\arxxxx & (\ar)+4(\arx)+6(\arxx)+4(\arxxx)+(\arxxxx) \\
\arxxnxx & 3(\ar)+8(\arx)+7(\arxx)+(\arxxn)+2(\arxxx)+2(\arxxnx)+(\arxxnxx) \\
\arxxnxnx & 5(\ar)+10(\arx)+6(\arxx)+2(\arxxn)+(\arxxx)+2(\arxxnx)+(\arxxnxn)+(\arxxnxnx) \\
\arxxxnx & 7(\ar)+10(\arx)+3(\arxx)+3(\arxxn)+3(\arxxnx)+(\arxxxn)+(\arxxxnx) \\
\arxxuxxn & 9(\ar)+12(\arx)+4(\arxx)+6(\arxxn)+4(\arxxnx)+(\arxxuxxn) \\
\arxxuxx & 7(\ar)+12(\arx)+6(\arxx)+4(\arxxn)+4(\arxxnxn)+(\arxxuxx)\\
\arxxxnxn & 9(\ar)+10(\arx)+3(\arxx)+3(\arxxn)+(\arxxx)+3(\arxxnxn)+(\arxxxnxn) \\
\arxxnxnxn & 11(\ar)+10(\arx)+2(\arxx)+6(\arxxn)+(\arxxnx)+2(\arxxnxn)+(\arxxxn)+(\arxxnxnxn) \\
\arxxnxxn & 13(\ar)+8(\arx)+(\arxx)+7(\arxxn)+2(\arxxnxn)+2(\arxxxn)+(\arxxnxx)\\
\arxxxxn & 15(\ar)+4(\arx)+6(\arxxn)+4(\arxxxn)+(\arxxxxn) \\
\end{array}$$
\caption{Polynomials $h(e_\alpha)$ with $\dim\alpha=4$.}
\label{polys4}
\end{figure}

\begin{figure}[p]
$$\begin{array}{c|c}
\alpha & h^{-1}(e_\alpha) \\
\hline
\arxxxx & (\ar)-4(\arx)+6(\arxx)-4(\arxxx)+(\arxxxx) \\
\arxxnxx & -(\ar)+3(\arxx)+(\arxxn)-2(\arxxx)-2(\arxxnx)+(\arxxnxx) \\
\arxxnxnx & -(\ar)+2(\arxx)+2(\arxxn)-(\arxxx)-2(\arxxnx)-(\arxxnxn)+(\arxxnxnx) \\
\arxxxnx & (\ar)-4(\arx)+3(\arxx)+3(\arxxn)-3(\arxxnx)-(\arxxxn)+(\arxxxnx) \\
\arxxuxxn & (\ar)+4(\arx)+4(\arxx)-2(\arxxn)-4(\arxxnx)+(\arxxuxxn) \\
\arxxuxx & -(\ar)+4(\arx)-2(\arxx)+4(\arxxn)-4(\arxxnxn)+(\arxxuxx)\\
\arxxxnxn & -(\ar)-4(\arx)+3(\arxx)+3(\arxxn)-(\arxxx)-3(\arxxnxn)+(\arxxxnxn) \\
\arxxnxnxn & (\ar)+2(\arxx)+2(\arxxn)-(\arxxnx)-2(\arxxnxn)-(\arxxxn)+(\arxxnxnxn) \\
\arxxnxxn & (\ar)+(\arxx)+3(\arxxn)-2(\arxxnxn)-2(\arxxxn)+(\arxxnxx)\\
\arxxxxn & -(\ar)-4(\arx)+6(\arxxn)+-4(\arxxxn)+(\arxxxxn) \\
\end{array}$$
\caption{Polynomials $h^{-1}(e_\alpha)$ with $\dim\alpha=4$.}
\end{figure}
\label{eulerpolys4}
\end{landscape}


\begin{thebibliography}{99}

% Special note to DR:  Don't annotate here.  Instead, write notes in the master bibliography.


\bibitem{barvinok_2007}
Alexander Barvinok.
Lattice Points, Polyhedra, and Complexity.
In: {\it Geometric Combinatorics.} 
Edited by Ezra Miller and Victor Reiner.
IAS/Park City Mathematics Series, vol.\ 14.
American Mathematical Society, 2007.

\bibitem{BR_2015}
Matthias Beck and Sinai Robins.
Computing the continuous discretely.
{\it Integer-point enumeration in polyhedra.}  2nd ed.\  
Springer, New York, 2015.

\bibitem{BV_1997}
Michel Brion and Mich\`ele Vergne.
Lattice points in simple polytopes.
{\it J.\ Amer.\ Math.\ Soc.} {\bf 10} (1997), no.\ 2, 371--392.

\bibitem{KSW_2007}
Yael Karshon, Shlomo Sternberg, and Jonathan Weitsman.
Exact Euler-Maclaurin formulas for simple lattice polytopes.
{\it Adv.\ in Appl.\ Math.} {\bf 39} (2007), no.\ 1, 1--50.
% The purpose of this paper is to present elementary proofs of the Euler-Maclaurin
% formulas for simple lattice polytopes.  I found this because 
% I am interested in obtaining a corresponding formula for generic orthotopes.
% These three coauthored two other articles with a similar theme.  One should
% compare with the work by Brion/Vergne, 1997.
%
% polytopes, toric geometry, general
% generic orthotopes

\bibitem{LM_2021}
Vadim Lozin and Mikhail Moshkov.
Critical Properties and Complexity Measures of Read-Once Boolean Functions.
{\it Ann.\ Math.] Artificial Intelligence}, {\bf 89} (2021) 595--614.

\bibitem{MS_2005}
Ezra Miller and Bernd Sturmfels.
{\it Combinatorial Commutative Algebra.}
Springer, Berlin, 2005.
% A graduate text on topological/geometrical methods in commutative algebra.  Includes
% a lot of ``staircase diagrams'' which are 3-dimensional orthotopes related to Newton
% polygons/polytopes.  Related to this is also some use of the secondary polytope of Gelfand and
% Kapranov.
% There is potential application of my Big Formula
% concerning floral arrangements to some of these ideas in this book....
%
% algebra, combinatorics
% general

\bibitem{richter_genericorthotopes}
David Richter.
Generic Orthotopes.
Submitted for publication, under revision.

\end{thebibliography}
\end{document}